\pdfoutput=1
\RequirePackage{rotating}
\documentclass[oneside,a4paper]{article}

\usepackage{amsbsy,amsfonts,amssymb,amsmath,amsthm}
\usepackage{dsfont}
\usepackage{rotating}
\usepackage{textcomp}

\normalsize
\newtheorem{theorem}{Theorem}
\newtheorem{proposition}{Proposition}
\newtheorem{lemma}{Lemma}
\newtheorem{corollary}{Corollary}

\newcommand{\keywords}[1]
           {\begin{center}
            \begin{minipage}{315.83pt}
            \small
            \noindent \emph{Keywords:}~{\textrm{#1}}
            \end{minipage}
            \end{center}
            \normalsize
           }
           
\newcommand{\ams}[2]
           {\begin{center}
            \begin{minipage}{315.83pt}
            \small
            \noindent \emph{2010 MSC:}~Primary {\uppercase{#1}}\\
            \phantom{\emph{2010 MSC:}~}Secondary {\uppercase{#2}}
            \end{minipage}
            \end{center}
            \par\normalsize
           }

\title{A general class of mosaic random fields} 

\author{%
Dimitri Schwab\footnote{Institut f\"{u}r Mathematik, Universit\"{a}t Mannheim, 68131 Mannheim, Germany, Email adress: dschwab@mail.uni-mannheim.de}
\and Martin Schlather\footnote{Institut f\"{u}r Mathematik, Universit\"{a}t Mannheim, 68131 Mannheim, Germany, Email adress: schlather@math.uni-mannheim.de}
\and J\"{u}rgen Potthoff\footnote{Institut f\"{u}r Mathematik, Universit\"{a}t Mannheim, 68131 Mannheim, Germany, Email adress: potthoff@math.uni-mannheim.de}
}

\date{\today}

\newcommand{\R}{\mathbb{R}}
\renewcommand{\S}{\mathbb{S}^2}
\newcommand{\1}{\mathds{1}}

\newcommand{\V}{\operatorname{Var}}
\newcommand{\sgn}{\mathop{\mathrm{sgn}}}
\newcommand{\N}{\mathbb{N}}
\newcommand{\I}{\mathbb{I}}
\newcommand{\E}{\mathbb{E}}
\DeclareMathOperator\artanh{artanh}
\newcommand\numberthis{\addtocounter{equation}{1}\tag{\theequation}}
\newlength{\mylength}
\newlength{\mylengthzwei}
\begin{document}

\maketitle

\begin{abstract}
We present a model of a random field on a topological space $M$ that unifies well-known models such as the Poisson hyperplane tessellation model, the random token model, and the dead leaves model. In addition to generalizing these submodels from $\R^d$ to other spaces such as the $d$-dimensional unit sphere $\mathbb{S}^d$, our construction also extends the classical models themselves, e.g.\ by replacing the Poisson distribution by an arbitrary discrete distribution. Moreover, the method of construction directly produces an exact and fast simulation procedure. By investigating the covariance structure of the general model we recover various explicit correlation functions on $\R^d$ and $\mathbb{S}^d$ and obtain several new ones.
\end{abstract}

\keywords{Random Field; Random Mosaic; Covariance Function; Simulation; Sphere} 


\ams{60G60}{60D05; 51M20; 60G15}

\section{Introduction}

A mosaic is a partitioning of some set $M$ into disjoint subsets $C_i$, called \emph{cells}. A corresponding random field $(Z(x),x\in M)$ can be obtained by assigning to each cell $C_i$ a random variable $V_i$ and setting $Z(x)=V_i$ for all $x\in C_i$. The random variables $V_i$ are not necessarily independent nor identically distributed. We call $Z$ a \emph{mosaic random field}. Usually the mosaic itself is also chosen to be random and independent of the random variables $V_i$. 

An important example is the mosaic build from a Poisson hyperplane tessellation in $\R^d$ \cite{Miles69, Mat72, Mat75} and the corresponding random field \cite{Chi09, Lantue02}. Let $\mathbb{S}^{d-1}_+$ denote the upper unit hemisphere, i.e.\ the set consisting of all $x\in\R^d$ with $\|x\|=1$ and $x_d\geq 0$. Given a Poisson point process $\Pi$ in $\mathbb{S}^{d-1}_+\times\R$, for each point $(x,r)$ of a realization of $\Pi$ a hyperplane with normal vector $\sgn(r) x$ pointing from the origin to the hyperplane and distance $|r|$ from the origin is drawn. The cells are the polytops delimited by this network of random hyperplanes and the random field is defined by assigning to each cell a different random variable from an independent and identically distributed (i.i.d.)\ sequence $(U_i,i\in\N)$. Due to the Poisson distributed number of hyperplanes this random field possesses an exponential covariance function.

Another well-known model is the random token model in $\R^d$ \cite{Chi09, Lantue02}. Here bounded subsets or tokens are placed at the points of a Poisson point process and to each token $B_i$ a random variable is associated from an i.i.d.\ sequence $(U_i,i\in\N)$. At each location $x$ the random field is then defined to be the sum of all random variables $U_i$ that are associated to tokens containing $x$. This model was first introduced as the random coin model by Sironvalle \cite{Siro80}, where the tokens are balls with random diameter, and it was used among others to model a random field with spherical covariance function \cite{Siro80, Chi09, Lantue02}. It can be viewed as a mosaic random field as well: If $n\in\N$ tokens are drawn, for any $I\subseteq\{1,\dots,n\}$ the cell $C_I$ on which the random field is constant consists of all the points contained in $\bigl(\bigcap_{i\in I}B_i\bigr)\bigcap\,\bigl(\bigcap_{j\in I^{c}}B_j^{c}\bigr)$, where $I^{c}$ denotes the complement of $I$ in $\{1,\dots,n\}$ and $B^{c}$ denotes the complement of $B$ in $M$. Putting it this way, to each cell $C_I$ the random variable $V_I=\sum_{i\in I} U_i$ is assigned so that the random variables associated to different cells might be dependent in contrast to the hyperplane random field. 

Mosaic random fields can be used to model phenomena that show a piecewise constant behavior over subdomains of $\R^d$. On the other hand, a suitable normalized sum of independent copies of a mosaic random field converges by the central limit theorem to a Gaussian random field with the same covariance function as the number of copies increases to infinity. In case of finite third-order absolute moments of the marginals, the Berry-Ess\'{e}en Theorem (e.g., \cite{Fel66}) provides an upper bound for the absolute difference between the marginal distribution of the normalized sum and a standard Gaussian distribution.

A convenient feature of mosaic random fields is their concrete structure which directly suggests a simulation procedure. In order to generate samples, it suffices to sample a finite number of independent random numbers from univariate distributions to build the underlying mosaic, and a finite number of random numbers for the values associated to the cells. Once this is done, the exact value of the simulation at each location $x$ can be computed by determining the cell that contains $x$. 

So far, many models and simulation methods have been developed in $\R^d$ (see \cite{PorMonSch12} for an overview), while for any other space $M$, e.g.\ the sphere, only few models are available \cite{BolLin11, Jun07, Jun08, LanSch15}. The two-dimensional sphere is of particular interest for applications. In geosciences, spatial data collected by satellites often cover a large portion of the globe (e.g., \cite{McPet96}) and the analysis of such data sets requires random fields and covariance models indexed by $\S$. Furthermore, random fields on the sphere serve as radial functions for star-shaped random sets \cite{Han15}. An advantage of mosaic models is that the construction of a mosaic random field does not require much specific structure of the underlying space and they are therefore also applicable to the sphere. 

The mosaics in the present paper are build by intersections of a random number of random sets in a topological space $M$ (cf.\ section \ref{sec:Model}). The procedure that assigns random variables to the cells of the mosaic determines the type of submodel. This sequential construction allows for a step-by-step investigation of the covariance structure of the resulting mosaic random field. The benefit of this is the ability to chose any combination of the three characteristics (i) random set, (ii) random number of sets, and (iii) assignment procedure, and observe the resulting correlation function and the dependence of the correlation function on the characteristics. This way, we obtain and recover a large number of very different correlation functions on bounded subsets of $\R^d$ and on $\mathbb{S}^d$ for the general mosaic random field. As an example, we find the generalized Cauchy correlation function for the mosaic random field in $\R^d$ or $\mathbb{S}^d$ that is constructed by letting the number of random sets follow a compound negative binomial distribution, taking half-spaces or hemispheres as random sets, and assigning i.i.d.\ random variables to the cells of the resulting mosaic.

\section{The Model}\label{sec:Model}

We consider a random field $Z=(Z(x),x\in M)$ on a second countable locally compact Hausdorff space $M$ equipped with its Borel $\sigma$-algebra. Let $N$ be an $\N_0$-valued random variable, not almost surely equal to zero, and $(U_{i,j},i,j\in\N)$ a doubly indexed i.i.d.\ sequence of real-valued random variables with finite variances. Let $(B_n,n\in\N)$ be an i.i.d.\ sequence of random closed sets in $M$ \cite{Mat75}. We assume that the family formed by $N$, $(U_{i,j},i,j\in\N)$, and $(B_n,n\in\N)$ is independent. The random variables $U$ and $B$ refer to a generic member of the sequences $(U_{i,j},i,j\in\N)$ and $(B_n,n\in\N)$, respectively.

Let $\mathcal{P}_n$ denote the power set of $\{1,\dots,n\}$. For every $n\in\N$ we define the family $(C_I, I\in\mathcal{P}_n)$ of disjoint random subsets of $M$ by
\begin{align*}
C_I=\Biggl(\bigcap_{i\in I}B_i\Biggr)\bigcap\,\Biggl( \bigcap_{j\in\{1,\dots,n\}\setminus I}B_j^{c}\Biggr).
\end{align*}
For $n=0$ we define $(C_I,I\in\mathcal{P}_0)=(C_{\emptyset})$ by $C_{\emptyset}=M$. We call $C_I$ a \emph{random cell of $M$}.

Let $\mathcal{P}^*(\N)$ denote the set consisting of all finite subsets of $\N$ and let a function $g:\mathcal{P}^*(\N)\rightarrow\N$ be given. Suppose $(\I_I,I\in\mathcal{P}_n)$, $n\in\N_0$, are families of elements of $\mathcal{P}^*(\N)$. We generalize the Poisson hyperplane tessellation model, the random token model, and the dead leaves model as follows:
\begin{align}\label{allgFeld}
Z(x)=\sum_{I\in\mathcal{P}_N}\Biggl(\sum_{j\in\I_I}U_{g(I),j}\Biggr)\1_{x\in C_I},\quad x\in M.
\end{align}
We call the random field $Z$ \emph{simple mosaic random field} and we write $Z_M$ instead of $Z$, when $g$ is an injection of $\mathcal{P}^*(\N)$ and $\I_I=\{1\}$ for all $I\in\mathcal{P}^*(\N)$. In this case there exists an i.i.d.\ sequence $\left(U_I,I\in\mathcal{P}^*(\N)\right)$ such that we have
\begin{align}\label{simplemosaic}
Z_M(x)\overset{d}{=}\sum_{I\in\mathcal{P}_N}U_I\,\1_{x\in C_I},\quad x\in M,
\end{align}
where the equality is in the sense of distribution. If $M=\R^d$, the sets $B_n$ are half-spaces determined by random hyperplanes in $\R^d$, and $N$ is taken to be Poisson distributed, then $Z_M$ is the Poisson hyperplane tessellation model in $\R^d$.

The choices $\I_I=I$, $I\in\mathcal{P}^*(\N)$, and $g\equiv 1$ in \eqref{allgFeld} lead to the field
\begin{align}\label{RCfeld}
Z_{RT}(x)\overset{d}{=}\sum_{I\in\mathcal{P}_N}\Biggl(\sum_{i\in I} U_{i}\Biggr)\mathds{1}_{x\in C_I},\quad x\in M,
\end{align}
where in this case $(U_i,i\in\N)$ is an i.i.d.\ sequence. Here each random closed set $B_n$ is associated with a random variable and to each point $x\in M$ we assign the sum of all random variables associated to random closed sets containing $x$. The random field $Z_{RT}$ is called \emph{random token field} if $M=\R^d$ and we keep the name for general $M$ as above. 

In a third example, we take $g$ from the simple mosaic random field and $(\I_I,I\in\mathcal{P}_n)$, $n\in\N_0$, from the random token field and get
\begin{align}\label{Mischungfeld}
Z_{MRT}=\sum_{I\in\mathcal{P}_N}\Biggl(\sum_{i\in I}U_{g(I),i}\Biggr)\1_{x\in C_I},\quad x\in M.
\end{align}
We call $Z_{MRT}$ \emph{mixture random field}.

In the \emph{dead leaves model} (e.g.,  \cite{Chi09,Lantue02}), the random sets $(B_n,n\in\N)$ are placed sequentially in $M$, partially overlapping previously placed random sets. The corresponding random field $Z_{DL}$ is defined at each $x\in M$ as $Z_{DL}(x)=U$ for the random variable $U$ associated to the latest random set covering $x$. In our setup, this random field corresponds to the choices $\I_I=\{1\}$ and $g(I)=\1_{I\neq\emptyset}\max I $ for all $I\in\mathcal{P}^*(\N)$, such that
\begin{align}\label{deadleaves}
Z_{DL}(x)\overset{d}{=}\sum_{I\in\mathcal{P}_N} U_{g(I)}\,\1_{x\in C_I},\quad x\in M,
\end{align}
for an i.i.d.\ sequence $(U_i,i\in\N_0)$. 

Realizations of different mosaic random fields on $[-1,1]^2$, on the sphere, and on the torus are illustrated in figure \ref{figureone}. Figure \ref{figuretwo} displays the weighted sum of $n=10,100$, and $200$ realizations of a mosaic random field on the sphere.

\begin{figure}[t]
\includegraphics[width=\textwidth]{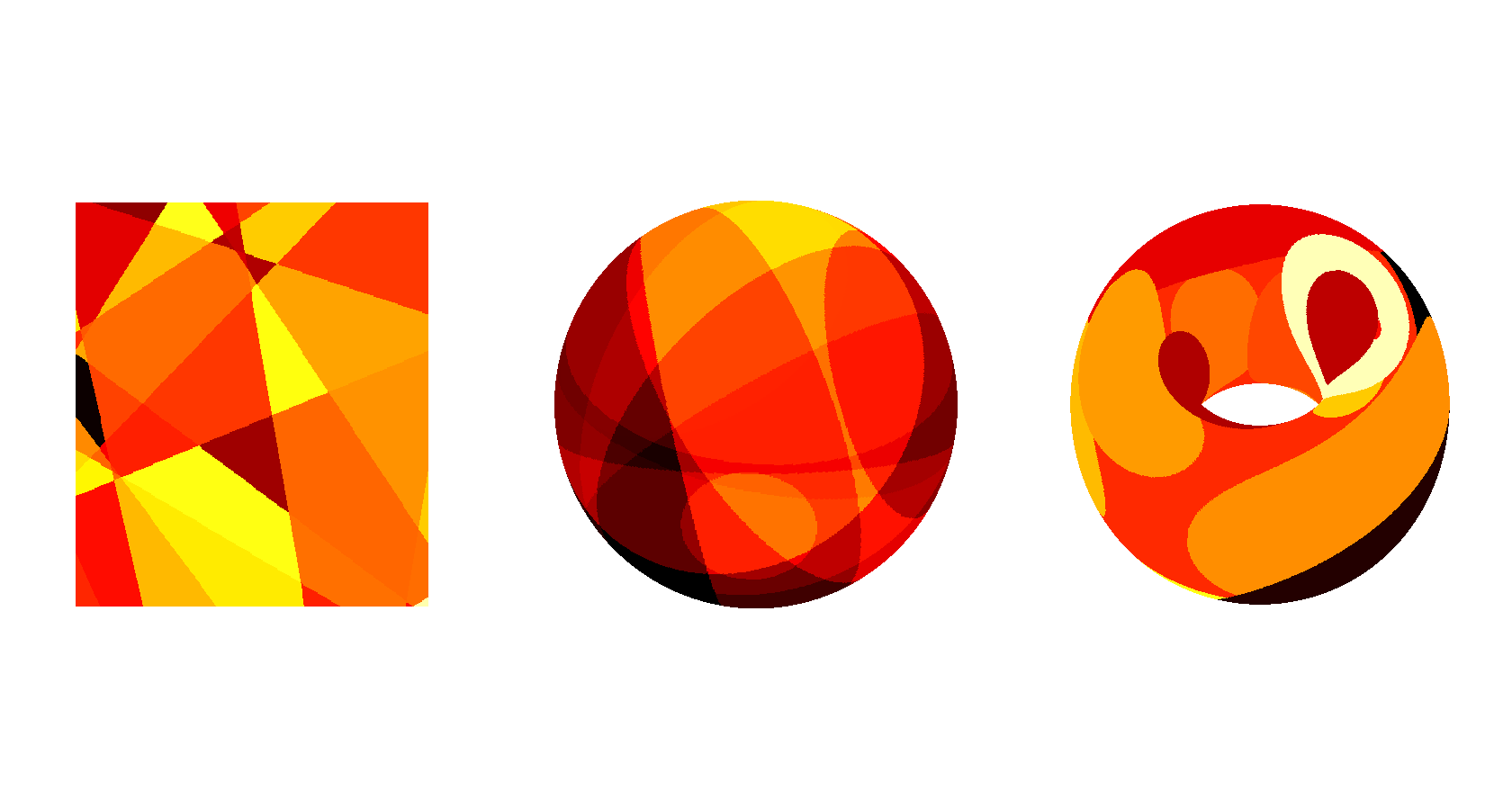}
\caption{From left to right: simple Mosaic with random half-spaces on $[-1,1]^2$, random token with random balls on the sphere, dead leaves with random balls on the torus.}\label{figureone}
\end{figure}

\begin{figure}[t]
\includegraphics[width=\textwidth, height=180px]{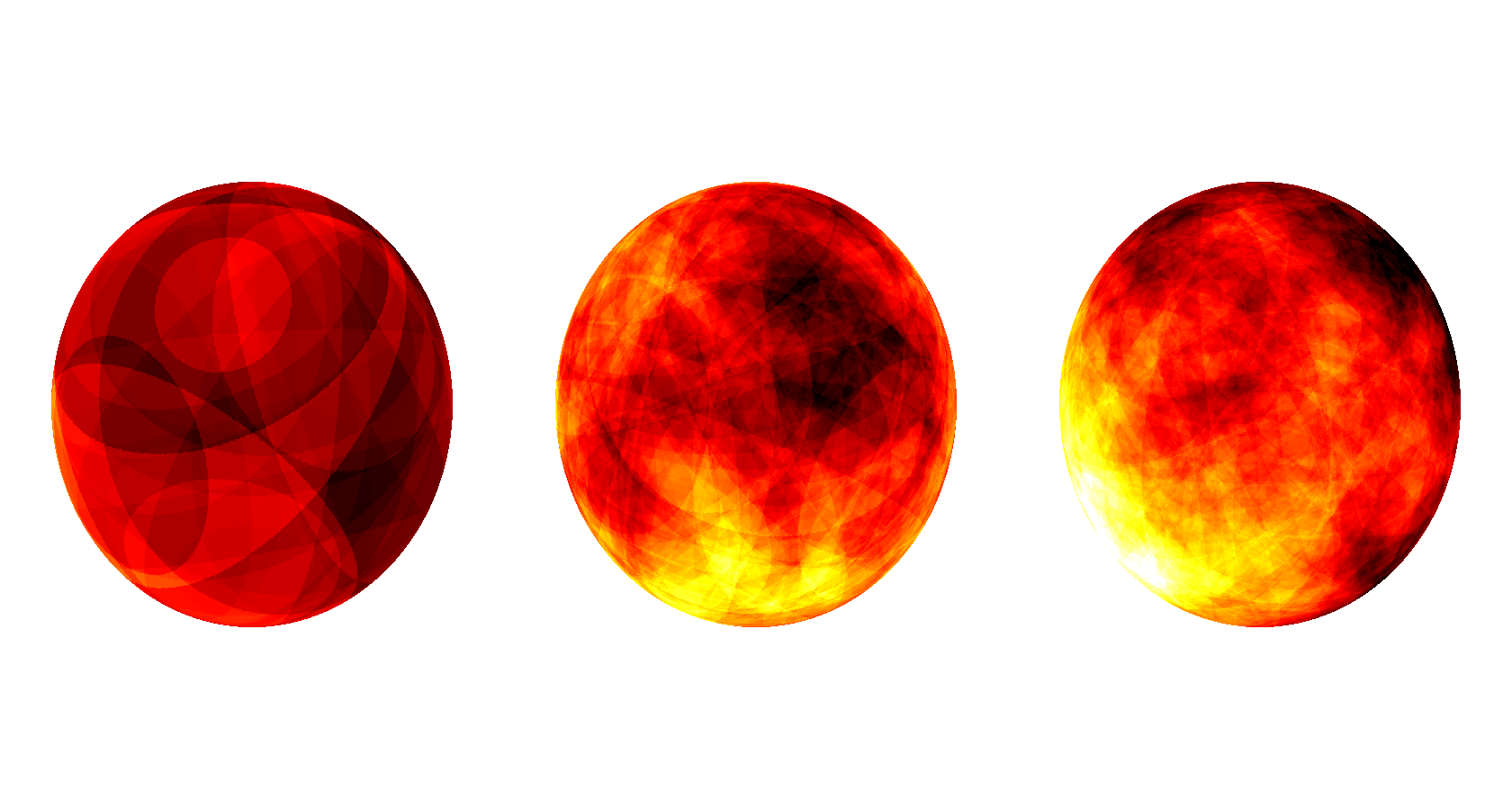}
\caption{Weighted sum of $n=10,100$, and $200$ realizations of a random token field on the sphere.}\label{figuretwo}
\end{figure}

In order to get reasonable analytic formulae for the covariance function of $Z$, we assume that there exist functions $f_n:\N_0^2\rightarrow\N_0$, $n\in\N_0$, such that for all $n\in\N_0$
\begin{align}\label{eq:assumption}
|\I_I\cap\I_J |=f_n\bigl(|I \cap J|, |I\bigtriangleup J|\bigr)\quad\text{for all }I,J\in\mathcal{P}_n,
\end{align}
holds, where $\bigtriangleup$ denotes the symmetric difference of two sets. In the following we present a class of functions $f_n$ for which we can construct families $(\I_I,I\in\mathcal{P}_n)$, $n\in\N_0$, such that \eqref{eq:assumption} holds. The functions corresponding to $Z_M, Z_{RT}, Z_{MRT}$, and $Z_{DL}$ above are given by $f_n(i,j)=1$, and $f_n(i,j)=i$, respectively, and they are included in the following class. 
\begin{lemma}
Suppose that $(a_n,n\in\N_0)$, $(b_n,n\in\N_0)$, and $(c_n,n\in\N_0)$ are sequences such that for every $n\in\N_0$, $a_n\in\mathbb{Z}$, $b_n,c_n\in\N_0$ holds. Assume furthermore that for all $n\in\N_0$, $a_n\geq -b_n$, $c_n\geq nb_n$ holds true, and set $f_n(i,j)=a_ni-b_nj+c_n$, $i,j\in\N_0$. Then there are families $(\I_I,I\in\mathcal{P}_n)$, $n\in\N_0$, such that \eqref{eq:assumption} holds.
\end{lemma}
\begin{proof}
Fix $n\in\N_0$. Let $A$ and $B_i,C_i$, $i=1,\dots,n$, be disjoint subsets of $\N$ such that $|A|=c_n-nb_n$, and $|B_i|=b_n$, $|C_i|=a_n+b_n$ holds for $i=1,\dots,n$. Set
\begin{align*}
\I_I=A\,\bigcup\,\Biggl(\bigcup_{i\in\{1,\dots,n\}\setminus I}B_i\Biggr)\bigcup \,\Biggl(\bigcup_{i\in I}C_i\Biggr),\quad I\in\mathcal{P}_n,
\end{align*}
then we get for all $I,J\in\mathcal{P}_n$
\begin{align*}
|\I_I\cap \I_J|&=|A|+\Biggl|\bigcup_{i\in\{1,\dots,n\}\setminus I}\bigcup_{j\in\{1,\dots,n\}\setminus J}\left(B_i\cap B_j\right)\,\Biggr|+\Biggl|\,\bigcup_{i\in I}\bigcup_{j\in J}\left(C_i\cap C_j\right)\,\Biggr|\\
&=|A|+\bigl(n-|I\cup J|\bigr)|B_1|+|I\cap J||C_1|=f_n\bigl(|I\cap J|,|I\bigtriangleup J|\bigr),
\end{align*}
and the lemma is proved.
\end{proof}

For $x,y\in M$ and $n\in\N_0$, set $p_x=P(x\in B)$, $p_{xy}=P(x,y\in B)$, and let $V_{xy,n}=(V_{xy,n}^1,V_{xy,n}^2,V_{xy,n}^3,V_{xy,n}^4)$ be a multinomial distributed random vector with parameters $n$, $p_{xy}$, $p_x-p_{xy}$, $p_y-p_{xy}$, and $1-p_x-p_y+p_{xy}$. In the case $n=0$, the vector $V_{xy,n}$ equals the zero vector almost surely.

\begin{theorem}\label{thmcovallg}
Suppose that there are functions $(f_n,n\in\N_0)$ such that \eqref{eq:assumption} holds for the families $(\I_I,I\in \mathcal{P}_n)$, $n\in\N_0$, of the random field $(Z(x),x\in M)$ defined in \eqref{allgFeld}. Then for all $x,y\in M$
\begin{align}\label{E-Wert}
\E Z(x)&=\E U \,\E f_N\bigl(V_{xx,N}^1,0\bigr)
\end{align}
and 
\begin{align*}\label{gem Moment}
\E Z(x)Z(y)&=\V U\,\E f_N\bigl(V_{xy,N}^1,V_{xy,N}^2+V_{xy,N}^3\bigr)\\
&\hspace{1.5em}+ (\E U)^2\,\E f_N\bigl(V_{xy,N}^1+V_{xy,N}^2,0\bigr)f_N\bigl(V_{xy,N}^1+V_{xy,N}^3,0\bigr)- G_{xy}\V U\numberthis
\end{align*}
with
\begin{align}\label{R}
G_{xy}=\E\sum_{\substack{I,J\in\mathcal{P}_N\\g(I)\neq g(J)}}P(x\in C_I,y\in C_J)f_N\bigl(| I\cap J|,| I\triangle J|\bigr)
\end{align}
holds true.
\end{theorem}

\begin{proof}
By definition of the cells, independence and identity of the distributions of the $B_i$, $i\in\N$, we have for every $n\in\N_0$, $I\in\mathcal{P}_n$, and every $x\in M$
\begin{align*}
P(x\in C_I)=\prod_{i\in I}P(x\in B_i)\prod_{j\in\{1,\dots,n\}\setminus I}P(x\notin B_j)=p_x^{| I |}(1-p_x)^{n-| I|}.
\end{align*}
Using this we get
\begin{align*}
\E Z(x)&=\sum_{n\in\N_0}P(N=n)\,\E\Biggl(\sum_{I\in\mathcal{P}_n}\Biggl(\sum_{j\in \I_I} U_{g(I),j}\Biggr)\1_{x\in C_I}\Biggr)\\
&=\sum_{n\in\N_0}P(N=n)\sum_{I\in\mathcal{P}_n}P(x\in C_I)\sum_{i\in \I_I}\E U_{g(I),i}\\
&=\E U \sum_{n\in\N_0} P(N=n)\sum_{I\in\mathcal{P}_n}p_x^{| I|}(1-p_x)^{n-| I|}|\I_I|.
\end{align*}
By assumption $|\I_I|=f_n(| I|,0)$, and as there are $\binom{n}{k}$ subsets of $\{1,\dots,n\}$ with $k$ elements, we find
\begin{align*}
\E Z(x)&=\E U\sum_{n\in\N_0}P(N=n)\sum_{k=0}^n \binom{n}{k}p_x^k (1-p_x)^{n-k} f_n(k,0)\\
&=\E U  \sum_{n\in \N_0}P(N=n)\,\E\bigl( f_n\bigl(V_{xx,n}^1,0\bigr) \big\vert N=n\bigr) = \E U \,\E f_N\bigl(V_{xx,N}^1,0\bigr),
\end{align*}
proving formula \eqref{E-Wert}. Regarding the mixed moment, we obtain
\begin{align*}
\sum_{i\in\I_I,j\in\I_J}\E U_{g(I),i}U_{g(J),j}=(\E U)^2\bigl(| \I_I| | \I_J|-| \I_I\cap \I_J|\bigr)+\sum_{i\in\I_I\cap\I_J }\E U_{g(I),i}U_{g(J),i}
\end{align*}
for all $I,J\in\mathcal{P}^*(\N)$. Since the second indices of the random variables in the last sum are equal, the last sum equals $\E U^2 | \I_I\cap\I_J|$ in case $g(I)=g(J)$ and $(\E U)^2 | \I_I\cap\I_J|$ otherwise. Hence
\begin{align*}
&\sum_{i\in\I_I}\sum_{j\in\I_J}\E U_{g(I),i}U_{g(J),j}=(\E U )^2\,| \I_I||\I_J|+\V U\,|\I_I\cap\I_J|\,\1_{g(I)=g(J)}.
\end{align*}
This yields
\begin{align*}
\E Z(x)Z(y)&=\sum_{n\in\N_0}P(N=n)\sum_{I,J\in\mathcal{P}_n}P(x\in C_I,y\in C_J)\sum_{i\in\I_I}\sum_{j\in\I_J}\E U_{g(I),i }U_{g(J),j}\\
  &=\V U\sum_{n\in\N_0}P(N=n)\sum_{I,J\in\mathcal{P}_n}P(x\in C_I,y\in C_J)\,|\I_I\cap\I_J|\\
 &\hspace{4em}+(\E U)^2\sum_{n\in\N_0}P(N=n)\sum_{I,J\in\mathcal{P}_n}P(x\in C_I,y\in C_J)\,|\I_I||\I_J|\\
 &\hspace{4em}-\V U\, G_{xy}. 
\end{align*}
Furthermore, for $n\in\N_0$,
\begin{align*}
 P&(x\in C_I,y\in C_J) \\
&=\, P(x,y\in B)^{| I\cap J|}P(x\in B,y\notin B)^{| I\setminus J|}P(x\notin B,y\in B)^{| J\setminus I|}P(x,y\notin B)^{n-| I\cup J|}\\
&=\, p_{xy}^{| I\cap J|}(p_x-p_{xy})^{| I\setminus J|}(p_y-p_{xy})^{| J\setminus I|}(1-p_x-p_y+p_{xy})^{n-| I\cup J|}.
\end{align*}
With the assumptions on $(\I_I,I\in\mathcal{P}_n)$, $n\in\N_0$, and the multinomial distribution we get
\begin{align*}\label{eq:grosseformel}
\sum_{I,J\in\mathcal{P}_n}P(x&\in C_I,y\in C_J)\,|\I_I\cap\I_J|\\
&=\sum_{\substack{k_1,k_2,k_3,k_4\in\N_0\\k_1+k_2+k_3+k_4=n}}\binom{n}{k_1,k_2,k_3,k_4}p_{xy}^{k_1}(p_x-p_{xy})^{k_2}(p_y-p_{xy})^{k_3}\\
&\hspace{8em}\times(1-p_x-p_y+p_{xy})^{k_4} f_n(k_1,k_2+k_3)\\
&=\E \bigl(f_n\bigl(V_{xy,n}^1,V_{xy,n}^2+V_{xy,n}^3\bigr)\big\vert N=n\bigr).\numberthis
\end{align*}
Similarly, the sum $\sum_{I,J\in\mathcal{P}_n}P(x\in C_I,y\in C_J)\,| \I_I||\I_J|$ reduces to the expression \eqref{eq:grosseformel} where $f_n\bigl(V_{xy,n}^1,V_{xy,n}^2+V_{xy,n}^3\bigr)$ is replaced by $f_n\bigl(V_{xy,n}^1+V_{xy,n}^2,0\bigr)f_n\bigl(V_{xy,n}^1+V_{xy,n}^3,0\bigr)$, yielding formula \eqref{gem Moment}.
\end{proof}

We write $\rho_i$, $i=M, RT, MRT$, and $DL$, for the correlation function of $Z_i$. Fur\-ther\-more, we let $\psi_N$ be the probability generating function of the $\N_0$-valued random variable $N$. 

\begin{corollary}\label{KorollarmitSPezialfallen}
Let $U$, $N$, and $B$ be such that $\V U>0$, $\E N >0$, and $p_x>0$ for all $x\in M$. Then for all $x,y\in M$
\begin{align}\label{corrmosaik}
\rho_M(x,y)=\psi_N(1+2p_{xy}-p_x-p_y),
\end{align}
\begin{align}\label{corrrc}
\rho_{RT}(x,y)=\frac{ap_{xy}+bp_xp_y}{\sqrt{(a+bp_x)(a+bp_y)p_xp_y}}
\end{align}
with $a=\E U^2\,\E N$ and $b=(\E U)^2(\V N-\E N)$,
\begin{align}\label{corrmischung}
\rho_{MRT}(x,y)&=\frac{p_{xy}\bigl(c\psi_N'(1+2p_{xy}-p_x-p_y)-d\bigr)}{\sqrt{(a+b p_x)(a+b p_y)p_xp_y}}+\rho_{RT}(x,y)
\end{align}
with $c=\V U$ and $d=\V U\,\E N$, and
\begin{align}\label{eq:dead leaves corr}
\rho_{DL}(x,y)=\frac{p_{xy}+(p_x+p_y-2p_{xy})\psi_N(1-p_x-p_y+p_{xy})}{p_x+p_y-p_{xy}}
\end{align}
hold true.
\end{corollary}
\begin{proof}
For the simple mosaic random field \eqref{simplemosaic} we have $\I_I=\{1\}$ for all $I\in\mathcal{P}^*(\N)$, hence the functions $f_n$ in \eqref{eq:assumption} can by taken to be identically $1$. Consequently, we get $\E Z_M(x)=\E U$ for all $x\in M$ from \eqref{E-Wert}. Since $g$ is injective for this field, we have for all $x,y\in M$ for the variable $G_{xy}$ from \eqref{R} with the same reasoning as in the proof of Theorem \ref{thmcovallg}
\begin{align*}
G_{xy}&=\sum_{n\in\N_0}P(N=n)\sum_{I\neq J}P(x\in C_I,y\in C_J)\\
&=1- \sum_{n\in\N_0}P(N=n)\sum_{I\in\mathcal{P}_n}P(x,y\in C_I)\\
&=1-\psi_N(1+2p_{xy}-p_x-p_y).
\end{align*}
Thus formula \eqref{gem Moment} yields
\begin{align*}
\E Z_M(x)Z_M(y)&=\V U\,\psi_N(1+2p_{xy}-p_x-p_y)+(\E U)^2.
\end{align*}
From this we can compute the variance of $Z_M(x)$, the covariance of $Z_M(x)$ and $Z_M(y)$, and then \eqref{corrmosaik} follows. In case of the random token field \eqref{RCfeld}, we have $\I_I=I$ for all $I\in\mathcal{P}^*(\N)$, and we can choose $f_n$ to be the projection on the first coordinate. Therefore
\begin{align*}
\E Z_{RT}(x)=\E U\,\E V_{xx,N}^1=\E U\, \E N p_x
\end{align*}
by Theorem \ref{thmcovallg}. The function $g$ is identically $1$ in case of the random token field, hence $G_{xy}=0$ and 
\begin{align*}
\E Z_{RT}(x)Z_{RT}(y)&=\V U\,\E V_{xy,N}^1+(\E U)^2\,\E \bigl(V_{xy,N}^1 
+V_{xy,N}^2\bigr)\bigl(V_{xy,N}^1+V_{xy,N}^3\bigr)
\end{align*}
by \eqref{gem Moment}. The covariance of the components of a multinomial distributed random vector is well-known and a straight forward computation yields
\begin{align*}\label{gemmomrc}
\E Z_{RT}(x)Z_{RT}(y)&=\E U^2\E N p_{xy}+(\E U)^2(\E N^2-\E N) p_x p_y\\
&=a p_{xy}+b p_xp_y+ \E Z_{RT}(x)\E Z_{RT}(y),\numberthis
\end{align*}
which implies \eqref{corrrc}. Now consider the mixture random field \eqref{Mischungfeld}. Again, we have $\I_I=I$ for all $I\in\mathcal{P}^*(\N)$ and we can choose the same $f_n$ as above. Consequently, $\E Z_{MRT}(x)=\E Z_{RT}(x)$. But in contrast to the random token field, $g$ is injective for the mixture random field. Reasoning as in the proof of Theorem \ref{thmcovallg} we obtain
\begin{align*}
G_{xy}&=\sum_{n\in\N_0}P(N=n)\sum_{I,J\in\mathcal{P}_n}P(x\in C_I,y\in C_J)| I\cap J|\\
&\hspace{4em}-\sum_{n\in\N_0}P(N=n)\sum_{I\in\mathcal{P}_n}P(x,y\in C_I)| I|\\
&=\E N p_{xy}-\sum_{n\in\N_0}P(N=n)\left(1+2p_{xy}-p_x-p_y\right)^{n-1} n p_{xy}\\
&=\E N p_{xy}-p_{xy}\psi_N'(1+2p_{xy}-p_x-p_y)
\end{align*}
and then with \eqref{gem Moment} and \eqref{gemmomrc}
\begin{eqnarray*}
\lefteqn{ \E Z_{MRT}(x)Z_{MRT}(y)}\\
 & & = \E Z_{RT}(x) Z_{RT}(y)-G_{xy} \V U \\
  & & = c p_{xy}\psi_N'(1+2p_{xy}-p_x-p_y)+(a-d) p_{xy}+ bp_xp_y+\E Z_{MRT}(x)\E Z_{MRT}(y).
\end{eqnarray*}   
This shows \eqref{corrmischung}. For the dead leaves model \eqref{deadleaves} we have $f_n\equiv 1$ for all $n\in \N_0$ and hence $\E Z_{DL}(x)=\E U$. In order to compute $G_{xy}$ we let  \begin{align*}
A_n=\sum_{\substack{I,J\in \mathcal{P}_n \\ g(I)=g(J)}}P(x\in C_I,y\in C_J),\quad n\in\N_0.
\end{align*}
Writing $\mathcal{P}_{n+1}=\mathcal{P}_n\cup\{I\cup\{n+1\}:I\in\mathcal{P}_n\}$ and using $g(I)=\1_{I\neq\emptyset}\max I$ we get the recurrence relation
\begin{align*}
A_{n+1}=P(x,y\notin B)A_n+P(x,y\in B),\quad n\in\N_0,
\end{align*}
which leads to 
\begin{align*}
A_n=\frac{P(x,y\in B)}{1-P(x,y\notin B)}+\frac{P(x\in B,y\notin B)+P(x\notin B,y\in B)}{1-P(x,y\notin B)}P(x,y\notin B)^n,\quad n\in\N_0,
\end{align*}
and then with \eqref{R}
\begin{align*}
G_{xy}&=1-\sum_{n\in\N_0}P(N=n)A_n\\
&=\frac{P(x\in B,y\notin B)+P(x\notin B,y\in B)}{1-P(x,y\notin B)}\bigl(1-\psi_N(P(x,y\notin B))\bigr).
\end{align*}
Collecting terms we get with \eqref{gem Moment}
\begin{align*}
\E Z_{DL}(x)Z_{DL}(y)=\E U^2-\V U \bigl(1-\psi_N(1-p_x-p_y+p_{xy})\bigr)\frac{p_x+p_y-2p_{xy}}{p_x+p_y-p_{xy}}
\end{align*}
and then \eqref{eq:dead leaves corr} follows.
\end{proof}
We end this section by considering the special case where $N$ is a Poisson random variable. In this case we write $\hat{\rho}_i$, $i=M,RT$, and $MRT$ for the correlation function of $Z_i$. Plugging in the moments and the probability generating function of the Poisson distribution into formulae \eqref{corrmosaik}, \eqref{corrrc}, and \eqref{corrmischung}, yields the relation
\begin{align*}
\hat{\rho}_{MRT}=\lambda\, \hat{\rho}_{RT}\,\hat{\rho}_M+(1-\lambda)\,\hat{\rho}_{RT}\quad\text{with}\quad \lambda =\frac{\V U}{\E U^2}\in(0,1].
\end{align*} 

\section{Explicit Formulae for bounded subsets M of $\R^d$}\label{sec:Rd}

The formulae in Corollary \ref{KorollarmitSPezialfallen} depend on the law of the random closed set $B$ through the probabilities $p_x=P(x\in B)$ and $p_{xy}=P(x,y\in B)$. Observe that for every $x\in M$ we have $p_x=p_{xx}$ so that it suffices to compute $p_{xy}$ for all $x,y\in M$. In what follows we give examples for $B$ and compute these probabilities to obtain explicit correlation functions. In order to get reasonable formulae we require that the random sets are in some sense uniformly placed in $\R^d$. In the pertinent literature this is typically done by placing the random sets at the points of a Poisson point process. The drawback of this method is that the number of random sets $N$ must follow a Poisson distribution. As the formulae in Theorem \ref{thmcovallg} and Corollary \ref{KorollarmitSPezialfallen} indicate, different distributions for $N$ may lead to different types of correlation functions, depending on the concrete choices determining a submodel. In the sequel we restrict ourselves to bounded subsets $M$ of $\R^d$,
and it is convenient - and without any serious loss of generality - to
assume furthermore that $M$ is closed or open. In this way it is possible to place the random sets uniformly on $M$ and have an arbitrary distribution on $\N_0$ for the number of random sets.

Let $\langle x,y\rangle$, $x,y\in\R^d$, denote the euclidean inner product in $\R^d$ and $\|x\|=\sqrt{\langle x,x\rangle}$, $x\in\R^d$, be the corresponding norm. Furthermore, $\mathcal{B}(\R^d)$ denotes the Borel $\sigma$-algebra on $\R^d$ and $\lambda^d$ the Lebesgue-measure. 

As a first example we take a half-space delimited by random hyperplanes for the random closed set $B$. For this, let $\mathbb{S}^{d-1}=\{z\in\R^{d}:\langle z,z\rangle=1\}$ be the $(d-1)$-dimensional unit sphere embedded in $\R^{d}$. The sphere $\mathbb{S}^0$ is just the set $\{-1,1\}$. For $d\geq 2$, it is convenient to use spherical coordinates for $\mathbb{S}^{d-1}$, which are given by the map $\phi_{d-1}:[0,2\pi)\times[0,\pi]^{d-2}\rightarrow\mathbb{S}^{d-1}$ recursively defined by
\begin{align*}\label{eq:sphaerische koordinaten}
\phi_1(\varphi)&=\bigl(\cos\varphi,\sin\varphi\bigr),\\
\phi_k(\varphi,\theta_1,\dots,\theta_{k-1})&=\bigl(\phi_{k-1}(\varphi,\theta_1,\dots,\theta_{k-2})\sin\theta_{k-1},\cos\theta_{k-1}\bigr), \quad k\geq 2.\numberthis
\end{align*}
In order for $\phi_{d-1}$ to be one-to-one, the domain of $\phi_{d-1}$ has to be restricted but this can be neglected for our purposes. For $d\geq 2$, let $\mathcal{B}(\mathbb{S}^{d-1})$ be the Borel $\sigma$-algebra on $\mathbb{S}^{d-1}$ and let $\sigma_{d-1}$ denote the surface measure of $\mathbb{S}^{d-1}$, which admits the representation
\begin{align}\label{eq:surface measure sphere}
\sigma_{d-1}(A)=\int_0^{2\pi}\int_0^\pi\dots\int_0^\pi \1_A\bigl(\phi_{d-1}(\varphi,\theta_1,\dots,\theta_{d-2})\bigr)\prod_{k=1}^{d-2}\sin^k\theta_k\, d\theta_{d-2}\dots d\theta_1 d\varphi
\end{align}
for every $A\in\mathcal{B}(\mathbb{S}^{d-1})$. The total mass of $\sigma_{d-1}$ is $2\pi^{d/2}/\Gamma(d/2)$ and we let $\hat{\sigma}_{d-1}=2^{-1}\pi^{-d/2}\Gamma(d/2)\sigma_{d-1}$ denote the uniform probability measure on $\mathbb{S}^{d-1}$.

A hyperplane $P(x,r)$ in $\R^d$, given in normal form, is the set of all $z\in\R^d$ with $\langle z,x\rangle=r$, where $x\in\mathbb{S}^{d-1}$, $r\in\R$, and $rx$ is the vector from the origin perpendicular to $P(x,r)$. The hyperplane $P(x,r)$ divides $\R^d$ into two half-spaces, consider the half-space that is given by $H(x,r)=\{ z\in\R^d : \langle z,x\rangle \geq r\}$. Let $(X_n,n\in\N)$ be an independent sequence of uniformly distributed random variables on $\mathbb{S}^{d-1}$ (e.g., \cite{Mull56, Mar72}) and let $(R_n,n\in\N)$ be an independent sequence of uniformly distributed random variables on the interval $[-C_M,C_M]$ for a constant $C_M>0$ large enough such that $M$ is contained in a closed ball with radius $C_M$ centered at the origin. Furthermore, let $(X_n,n\in\N)$ and $(R_n,n\in\N)$ be independent. Then $(H_n,n\in\N)$ defined by $H_n=H(X_n,R_n)\cap M$ is a sequence of random closed sets in $M$. 

For the second example we fix $a>0$ and let $(Y_n,n\in\N)$ be an independent sequence of random variables, uniformly distributed on the ball $B_{C_M+a/2}(0)$ of radius $C_M+a/2$ centered at the origin. Furthermore, let $(D_n,n\in\N)$ be an i.i.d.\ sequence of $[0,a]$-valued random variables, independent of $(Y_n,n\in\N)$. Then,  $B_n=B_{D_n/2}(Y_n)\cap M=\{z\in M : \|z-Y_n\|\leq D_n/2\}$ defines an i.i.d.\ sequence of random closed sets in $M$. Since $Y$ is uniformly distributed and independent of the diameter $D$, we have 
\begin{align}\label{eq:pxyinb ist covariogram}
P\bigl(x,y\in B_{D/2}(Y)\bigr)&=P\bigl(Y\in B_{D/2}(x)\cap B_{D/2}(y)\bigr)=\frac{\E \lambda^d\bigl(B_{D/2}(x)\cap B_{D/2}(y)\bigr)}{\lambda^d\bigl(B_{C_M+a/2}(0)\bigr)}
\end{align}
for all $x,y\in M$. If for example $D$ is taken to be deterministic, this reduces to a normalized geometric covariogram of a ball (e.g., \cite{Lantue02}). The intersection of two balls in $\R^d$ can be represented as the union of two equally sized hyperspherical caps. Hence, if $D$ is equal to some $0<t\leq a$ it follows from \eqref{eq:pxyinb ist covariogram} and \cite{Li11} that
\begin{align}\label{eq:pxyinb incomplete beta}
P\bigl(x,y\in B_{t/2}(Y)\bigr)&=\frac{\Gamma(d/2+1)t^d}{\sqrt{\pi}(2C_M+a)^d\Gamma((d+1)/2)}B_{1-d_{xy}^2/t^2}\biggl( \frac{d+1}{2},\frac{1}{2}\biggr)\1_{d_{xy}\leq t},
\end{align}
where we define $d_{xy}=\|x-y\|$ and
\begin{align*}
B_x(a,b)=\int_0^x t^{a-1}(1-t)^{b-1}\, dt,\quad x\in[0,1], a,b>0,
\end{align*}
is the incomplete Beta function. For example in dimension $2$, we can use formulae $8.17.20$ in \cite{DLMF} and $8.391, 9.121.26$ in \cite{Ry65} to obtain
\begin{align*}
P\bigl(x,y\in B_{t/2}(Y)\bigr)&=\frac{2}{\pi(2C_M+a)^2}\biggl(t^2\arccos\frac{d_{xy}}{t}-d_{xy}\sqrt{t^2-d_{xy}^2}\biggr)\1_{d_{xy}\leq t}.
\end{align*}
If the diameter $D$ is chosen to be a continuously distributed random variable, equation \eqref{eq:pxyinb incomplete beta} has to be integrated with respect to the distribution of $D$. Sironvalle showed in \cite{Siro80}, that for $d=2$ the choice 
\begin{align}\label{eq: Alfaros distribution function}
F(x)=\frac{1}{a}\Bigl(a-\sqrt{a^2-x^2}\Bigr)\1_{0\leq x\leq a}+\1_{x>a},\quad x\in\R,
\end{align}
for the distribution function of the diameter $D$ results in $P(x,y\in B)$ being proportional to the spherical correlation function
\begin{align}\label{eq:spherical correlation}
\rho(x,y)=\biggl(1-\frac{3d_{xy}}{2a}+\frac{d_{xy}^3}{2a^3}\biggr)\1_{d_{xy}\leq a}.
\end{align}
In Proposition \ref{pro: WKen auf R^2} below we consider the case of uniformly distributed diameter.

An example for random sets which lead to a stationary but anisotropic correlation function is given by hyperrectangles of the form $E_n=E(Z_n)\cap M=\bigl\{z\in M: |z_1-Z_n^1|\leq a_1,\dots,|z_d-Z_n^d|\leq a_d\bigr\}$ for $a_1,\dots,a_d>0$ and an i.i.d.\ sequence $(Z_n,n\in\N)$ such that $Z=\bigl(Z^1,\dots,Z^d\bigr)$ is uniformly distributed on $\prod_{k=1}^{d}[-(R_k+a_k),R_k+a_k]$ where $R=\prod_{k=1}^d[-R_k,R_k]$ is a hyperrectangle large enough such that $M\subseteq R$.

\begin{proposition}\label{pro: WKen auf R^2}
Suppose that $M\subset \R^d$ is as above,  fix $x,y\in M$, and let $d_{xy}=\|x-y\|$. Then for  $H=H(X,R)\cap M$
\begin{align}\label{eq:px pxy rd}
P(x,y\in H)=\frac{1}{2}-\frac{1}{\Omega_1}d_{xy}
\end{align}
holds with $\Omega_1=4\sqrt{\pi}C_M\Gamma((d+1)/2)/\Gamma(d/2)$. For $B=B_{D/2}(Y)\cap M$ with $D$ being uniformly distributed on $[0,a]$ the following formula holds true
\begin{align}\label{eq: R^d uniform radius}
P(x,y\in B)=\frac{1}{\Omega_2}\biggl(a^d B_{1-d_{xy}^2/a^2}\biggl(\frac{d+1}{2},\frac{1}{2}\biggr)-\frac{d_{xy}^{d+1}}{a}B_{1-d_{xy}^2/a^2}\biggl(\frac{d+1}{2},-\frac{d}{2}\biggr)\biggr)\1_{d_{xy}\leq a},
\end{align}
where $d_{xy}^{d+1}B_{1-d_{xy}^2/a^2}\left((d+1)/2,-d/2\right)$ is defined as zero for $d_{xy}=0$, and the constant is $\Omega_2=(d+1)\sqrt{\pi}(2C_M+a)^d\Gamma((d+1)/2)/\Gamma(d/2+1)$. For $E=E(Z)\cap M$
\begin{align}\label{eq:R^d quadrat token}
P(x,y\in E)=\prod_{k=1}^d\frac{1}{2(R_k+a_k)}\left(2a_k-|x_k-y_k|\right)_+
\end{align}
holds true.
\end{proposition}
\begin{proof}
The point $(x,r)\in\mathbb{S}^{d-1}\times\R$ defines the same hyperplane as the point $(-x,-r)$, but due to the opposite direction of the normal vector $x$, the relation $H(x,r)\setminus P(x,r)=H(-x,-r)^{c}$ holds true for the half-spaces. By construction, $X$ and $R$ have the same distribution as $-X$ and $-R$, respectively. Thus 
\begin{align*}
P\bigl(x\in H(X,R)\bigr)&=P\bigl(x\in H(-X,-R)\bigr)\\
&=P\bigl(x\in H(-X,-R)\setminus P(-X,-R)\bigr)=P\bigl(x\notin H(X,R)\bigr),
\end{align*}
which implies $P(x\in H)=P(x\notin H)=1/2$ for all $x\in M$. Now let $x\neq y$ and $d\geq 2$, then
\begin{align*}
P(x\in H,y\notin H)&=P\bigl(\langle X,y\rangle < R\leq \langle X,x\rangle\bigr)\\
&=\int_{\mathbb{S}^{d-1}}\frac{1}{2 C_M}\bigl(\langle z,x\rangle -\langle z,y\rangle\bigr)\1_{\langle z,x\rangle > \langle z,y\rangle}\,d\sigma_{d-1}(z)\\
&=\frac{d_{xy}}{2C_M}\E\biggl\langle X,\frac{x-y}{d_{xy}}\biggr\rangle\1_{\bigl\langle X,\frac{x-y}{d_{xy}}\bigr\rangle > 0}.
\end{align*}
Let $\mathcal{R}$ be a rotation which maps $(x-y)/d_{xy}\in\mathbb{S}^{d-1}$ to the point $\bigl(0,\dots,0,1\bigr)$, then $\langle X,(x-y)/d_{xy}\rangle =\bigl\langle \mathcal{R} X , \bigl(0,\dots,0,1\bigr)\bigr\rangle$. Since $\mathcal{R}X$ and $X$ have the same distribution, we have using \eqref{eq:sphaerische koordinaten} and \eqref{eq:surface measure sphere}
\begin{align*}
\E \biggl \langle X,\frac{x-y}{d_{xy}}\biggr\rangle \1_{\bigl\langle X,\frac{x-y}{d_{xy}}\bigr\rangle > 0}=\frac{1}{2\pi}\int_{0}^{2\pi}\sin \varphi\, \1_{\sin\varphi >0}\,d\varphi=\frac{1}{\pi}
\end{align*}
for $d=2$ and for $d\geq 3$ 
\begin{align*}
\E \biggl \langle X,\frac{x-y}{d_{xy}}&\biggr\rangle \1_{\bigl\langle X,\frac{x-y}{d_{xy}}\bigr\rangle > 0} \\
&=\frac{\Gamma(d/2)}{2\pi^{d/2}}\int_0^{2\pi}\int_0^{\pi}\dots\int_0^{\pi}\cos\theta_{d-2}\prod_{k=1}^{d-2}\sin^k\theta_k\,\1_{\cos\theta_{d-2}>0}\,d\theta_{d-2}\dots d\theta_1 d\varphi\\
&=\frac{\Gamma(d/2)}{2\sqrt{\pi}\Gamma((d+1)/2)},
\end{align*}
where we used formulae $3.621.1, 3.621.5,8.384.1$, and $8.335.1$ in \cite{Ry65}. For $d=1$, one can do the same computation without spherical coordinates since the uniform distribution on $\mathbb{S}^0$ is just the two-point distribution on $\{-1,1\}$ which assigns both values probability $1/2$. For $x=y$, the probability $P(x\in H,y\notin H)$ is zero. Hence, we have for all $x,y\in M$ and $d\geq 1$
\begin{align*}
P(x\in H,y\notin H)=\frac{\Gamma(d/2)}{4\sqrt{\pi} C_M \Gamma((d+1)/2)}\,d_{xy}
\end{align*}
and formula \eqref{eq:px pxy rd} is then obtained from $P(x,y\in H)=P(x\in H)-P(x\in H, y \notin H)$.

In the case of the random set $B_{D/2}(Y)\cap M$, it follows from \eqref{eq:pxyinb ist covariogram} and \eqref{eq:pxyinb incomplete beta} that
\begin{align*}
P\bigl(x,y\in B_{D/2}(Y)\cap M\bigr)&=\frac{\1_{d_{xy}\leq a}}{a\Omega_2}\int_{d_{xy}}^a (d+1)t^dB_{1-d_{xy}^2/t^2}\biggl(\frac{d+1}{2},\frac{1}{2}\biggr)\, dt\\
&=\frac{\1_{d_{xy}\leq a}}{a\Omega_2}\int_{d_{xy}}^a \int_0^{1-d_{xy}^2/t^2}(d+1)t^d s^{(d-1)/2}(1-s)^{-1/2} \,ds\, dt.
\end{align*}
An application of Fubinis theorem yields for the integral
\begin{align*}
\int_0^{1-d_{xy}^2/a^2}&s^{(d-1)/2}(1-s)^{-1/2}\int_{d_{xy}/\sqrt{1-s}}^a (d+1)t^d \,dt \, ds\\
&=a^{d+1}B_{1-d_{xy}^2/a^2}\biggl(\frac{d+1}{2},\frac{1}{2}\biggr)-\int_0^{1-d_{xy}^2/a^2}d_{xy}^{d+1}s^{(d-1)/2}(1-s)^{-(d+2)/2}\,ds.
\end{align*}
The last integral is $0$ if $d_{xy}=0$, and we can write it as $d_{xy}^{d+1}B_{1-d_{xy}^2/a^2}\left((d+1)/2,-d/2\right)$ if $0<d_{xy}\leq a$. Collecting terms we obtain \eqref{eq: R^d uniform radius}.

Regarding \eqref{eq:R^d quadrat token}, the components of $Z=\bigl(Z^1,\dots,Z^d\bigr)$ are independent and uniformly distributed on $[-(R_k+a_k),R_k+a_k]$ and \eqref{eq:R^d quadrat token} follows from 
\begin{align*}
P(x,y\in E)=\prod_{k=1}^dP\bigl(|x_k-Z^k|\leq a_k, |y_k-Z^k|\leq a_k\bigr)
\end{align*}
and
\begin{align*}
P\bigl(|x_k-Z^k|\leq a_k, |y_k-Z^k|\leq a_k\bigr)=\frac{1}{2(R_k+a_k)}(2a_k-|x_k-y_k|)_+
\end{align*}
for $k=1,\dots,d$.
\end{proof}

Taking for example $d=2$ we obtain from \eqref{eq: R^d uniform radius} with $8.391$ in \cite{Ry65} and $7.3.2.210$ in \cite{Prud86}
\begin{align*}\label{eq: pxyinb formel gleichverteilt d gleich 2}
P\bigl(x,y\in B\bigr)&=\frac{2}{3\pi(2 C_M+a)^2}\Biggl(a^2\arccos\frac{d_{xy}}{a}\\
&\hspace{4em}-2d_{xy}\sqrt{a^2-d_{xy}^2}+\frac{d_{xy}^3}{a}\artanh\sqrt{1-\frac{d_{xy}^2}{a^2}}  \Biggr)\1_{d_{xy}\leq a}.\numberthis
\end{align*}

Examples of correlation functions of mosaic random fields on bounded subsets $M$ of $\R^2$ which can be obtained from the combination of Proposition \ref{pro: WKen auf R^2} and Corollary \ref{KorollarmitSPezialfallen} are given in table~\ref{table R}. There we used repeatedly the fact, that the probability generating function of a compound random variable of the form $N=\sum_{l=1}^L K_l$ with $\N_0$-valued and independent random variables $L, K_1, K_2,\dots$ and identically distributed $K,K_1,K_2,\dots$, is given by the composition of the probability generating functions of $L$ and $K$. Furthermore we used the following result.

\begin{lemma}
For every $\alpha\in (0,1]$ there exists an $\N$-valued random variable $K$ such that the probability generating function $\psi_K$ of $K$ is given by
\begin{align}
\label{eq:pgf alpha power}
\psi_K(t)=1-(1-t)^{\alpha},\quad t\in[-1,1].
\end{align}
\end{lemma}

\begin{proof}
If $\alpha=1$, we can take a random variable $K$ which is almost surely equal to~$1$. In case $0<\alpha<1$, we define the distribution of $K$ by the probability mass function 
\begin{align}\label{Ableitungen}
p_k=-\frac{\Gamma(k-\alpha)}{\Gamma(-\alpha)\Gamma(k+1)},\quad k\geq 1.
\end{align}
Using the functional equation of the Gamma function we obtain $p_1=\alpha$ and $p_k=\alpha(1-\alpha)\dots (k-1-\alpha)/k!$ for all $k\geq 2$, thus $0\leq p_k\leq 1$ for all $k\geq 1$. By formula $7.3.1.27$ in \cite{Prud86} we have for all $t\in[-1,1]$ \begin{align*}
\sum_{k=1}^{\infty} p_k t^k=1-\sum_{k=0}^{\infty}\frac{\Gamma(k-\alpha)}{\Gamma(-\alpha)}\frac{t^k}{k!}=1-(1-t)^{\alpha},
\end{align*}
hence equations \eqref{Ableitungen} define a probability measure on $\N$ and the probability generating function of this measure is given by \eqref{eq:pgf alpha power}.
\end{proof}

\begin{sidewaystable}
\begin{center}
\renewcommand{\arraystretch}{1.5}
\begin{tabular}{l l l l l l}
\hline
Correlation function on $M\subseteq\R^2$ & Model  & $B$ & $N$ & $U$ & Lit. \\ \hline 
1. $\rho(x,y)=e^{-\left(d_{xy}/c_1\right)^{\alpha}}$ & SM  & $H(X,R)$ & $\sum_{l=1}^{L_1} K_l$ & &  \cite{Chi09, Lantue02} \\
2. $\rho(x,y)=\bigl(1-\bigl(\frac{d_{xy}}{c_2}\bigr)^{\alpha}\bigr)^n$ & SM & $H(X,R)$ & $\sum_{l=1}^{L_2} K_l$ & &\cite{PorMonSch12}* \\
3. $\rho(x,y)=\bigl(1+\bigl(\frac{d_{xy}}{c_1}\bigr)^{\alpha}\bigr)^{-\beta/\alpha}$ & SM  & $H(X,R)$ & $\sum_{l=1}^{L_3} K_l$ & & \cite{GneiSchl04} \\
4. $\rho(x,y)=\lambda_1\bigl(1-\frac{d_{xy}}{\pi C_M}\bigr)+(1-\lambda_1)$  & RT  & $H(X,R)$ & $\operatorname{Geo}\bigl(\frac{\lambda_1}{2(2-\lambda_1)}\bigr)$ & $\mathcal{N}\bigl(\frac{1}{\sqrt{2}},\frac{1}{2}\bigr)$ & \cite{Chi09, Lantue02}*\\
5. $\rho(x,y)=\lambda_2 \bigl(1-\frac{d_{xy}}{\pi C_M}\bigr)e^{-d_{xy}/c_1}+(1-\lambda_2)\bigl(1-\frac{d_{xy}}{\pi C_M}\bigr)$ & MRT  & $H(X,R)$ & $\operatorname{Poi}\bigl(\frac{\pi C_M}{c_1}\bigr)$ & $\mathcal{N}\bigl(1,\frac{\lambda_2}{1-\lambda_2}\bigr)$ & \cite{Chi09, Lantue02}*\\
6. $\rho(x,y)=1-2^{1-\alpha}\frac{d_{xy}/(\pi C_M)}{\bigl(1+d_{xy}/(\pi C_M)\bigr)^{1-\alpha}}$ & DL  & $H(X,R)$ & $K$ & & \eqref{eq:dead leaves corr} \\
7. $\rho(x,y)=\Bigl(\frac{2}{\pi}\arccos\frac{d_{xy}}{a}-\frac{2 }{\pi a^2}d_{xy}\sqrt{a^2-d_{xy}^2}\Bigr)\1_{d_{xy}\leq a}$ & RT  &  $B_{a/2}(Y)$ & $\operatorname{Poi}$ & &  \cite{Chi09, Lantue02} \\
8. $\rho(x,y)=\Bigl(1-\frac{3d_{xy}}{2a}+\frac{d_{xy}^3}{2a^3}\Bigr)\1_{d_{xy}\leq a}$ & RT  & $B_{D_1/2}(Y)$ & $\operatorname{Poi}$ & & \cite{Siro80}\\
9. $\rho(x,y)=\biggl(\frac{2}{\pi}\arccos\frac{d_{xy}}{a}-\frac{4}{\pi a^2}d_{xy}\sqrt{a^2-d_{xy}^2}$ & RT  & $B_{D_2/2}(Y)$ & $\operatorname{Poi}$ & & \eqref{eq: pxyinb formel gleichverteilt d gleich 2} \\
\hspace{\mylength}\hspace{3em}$+\frac{2}{\pi a^3}d_{xy}^3\artanh\sqrt{1-\frac{d_{xy}^2}{a^2}}\biggr)\1_{d_{xy}\leq a}$   \\
10. $\rho(x,y)=\frac{1}{4a_1a_2}\left(2a_1-| x_1-y_1|\right)_+\left(2a_2-| x_2-y_2|\right)_+$ & RT  & $E(Z)$ & $\operatorname{Poi}$ &  & \cite{Chi09, Lantue02}*  \\
\hline
\end{tabular}
\caption{Examples of correlation functions of mosaic random fields. Here $x$ and $y$ are points in a bounded subset $M$ of $\R^2$ and $d_{xy}=\|x-y\|$. The distributions of the random variables are as follows: $X\sim\mathcal{U}(\mathbb{S}^1)$, $R\sim\mathcal{U}\bigl([-C_M,C_M]\bigr)$ for a $C_M>0$ such that $M\subseteq B_{C_M}(0)$, $Y\sim\mathcal{U}\bigl(B_{C_M+a/2}(0)\bigr)$, $D_1$ has the distribution function \eqref{eq: Alfaros distribution function}, $D_2\sim\mathcal{U}\bigl([0,a]\bigr)$, $Z\sim\mathcal{U}\bigl(\prod_{k=1}^2 [-(R_k+a_k),R_k+a_k]\bigr)$ for $R_1,R_2>0$ such that $M\subseteq \prod_{k=1}^2[-R_k,R_k]$, the random variables $K,K_l$ are i.i.d.\ with the probability generating function defined in \eqref{eq:pgf alpha power}, $L_1\sim\operatorname{Poi}\bigl((\pi C_M/c_1)^{\alpha}\bigr)$, $L_2\sim\operatorname{Bin}\bigl(n,(\pi C_M/c_2)^{\alpha}\bigr)$, $L_3\sim\operatorname{NegBin}\bigl(\beta/\alpha,(1+(\pi C_M/c_1)^{\alpha})^{-1}\bigr)$. The parameters have the following range: $\alpha\in(0,1]$, $\beta,c_1,a,a_1,a_2>0$, $c_2\geq \pi C_M$, $\lambda_1\in(0,1]$, $\lambda_2\in(0,1)$, $n\in\N$. Wherever an entry is left blank the distribution of the corresponding random variable is arbitrary. The symbol $\operatorname{Poi}$ indicates, that the corresponding random variable is Poisson distributed, but the parameter of the Poisson distribution is arbitrary. A  \textquotesingle *\textquotesingle{} at the reference indicates that the given correlation function is new, but can be obtained as convex combinations or products of known correlation functions.} \label{table R}
\end{center}
\end{sidewaystable}

\section{Explicit Formulae for $M=\S$}\label{sec:sphere}

In this section we let $M=\mathbb{S}^d$ be the $d$-dimensional unit sphere, $\sigma_d$ the surface measure on $\mathbb{S}^d$ defined in \eqref{eq:surface measure sphere}, and $\phi_d$ the spherical coordinate map defined in \eqref{eq:sphaerische koordinaten}. Furthermore, we denote the geodesic metric or great circle metric on $\mathbb{S}^{d}$ by $d_{\mathbb{S}^d}(x,y)=\arccos\langle x,y\rangle$, $x,y\in\mathbb{S}^d$. 

Let $B_r(x)=\{z\in\mathbb{S}^d: d_{\mathbb{S}^d}(x,z)\leq r\}$ denote a closed ball or spherical cap on $\mathbb{S}^d$, centered at $x\in\mathbb{S}^d$ and with radius $r\in[0,\pi]$. Let $(X_n,n\in\N)$ be an independent sequence of random variables uniformly distributed on $\mathbb{S}^d$ (e.g., \cite{Mull56, Mar72}) and let $(R_n,n\in\N)$ be an i.i.d.\ sequence of random variables with values in $[0,\pi]$, independent of $(X_n,n\in\N)$. Then $B_n=B_{R_n}(X_n)$ defines an i.i.d.\ sequence of random closed sets in $\mathbb{S}^d$. As in the previous section, we have
\begin{align*}
P(x,y\in B)=\frac{\Gamma((d+1)/2)}{2\pi^{(d+1)/2}}\E\sigma_d\bigl(B_R(x)\cap B_R(y)\bigr),
\end{align*}
i.e.\ $P(x,y\in B)$ is proportional to the mean surface volume of the intersection of two spherical caps with random but equal radius.

For a deterministic radius $R=r\in[0,\pi]$ and $d=1$, an elementary geometric consideration yields
\begin{align*}
P\bigl(x,y\in B_r(X)\bigr)=\biggl(\frac{r}{\pi}-\frac{d_{\mathbb{S}^1}(x,y)}{2\pi}\biggr)_+,\quad x,y\in\mathbb{S}^1.
\end{align*}
Tovchigrechko and Vakser \cite{Tov01} used spherical trigonometry to obtain a formula for $\sigma_d\bigl(B_r(x)\cap B_r(y)\bigr)$ in case $d=2$, which results in
\begin{align*}\label{eq: Tovchi formel}
P\bigl(x,y\in B_r(X)\bigr)=\biggl(&\frac{1}{2\pi}\arccos\frac{\cos^2 r-\cos d_{\mathbb{S}^2}(x,y)}{\sin^2 r}\\
&\hspace{2em}-\frac{\cos r}{\pi}\arccos\frac{\cos r(1-\cos d_{\mathbb{S}^2}(x,y))}{\sin r\sin d_{\mathbb{S}^2}(x,y)}\biggr)\1_{d_{\mathbb{S}^2}(x,y)\leq 2r}\numberthis
\end{align*}
for all $x\neq y\in\mathbb{S}^2$ and $r\in(0,\pi/2]$. For higher dimension, Estrade and Istas \cite{Est10} provide the recursive formula 
\begin{align}\label{eq:Estrade}
\sigma_d\bigl(B_r(x)\cap B_r(y)\bigr)=\int_{-\sin r}^{\sin r}(1-a^2)^{(d-2)/2}\sigma_{d-1}\bigl(B_{r(a)}(x')\cap B_{r(a)}(y')\bigr)\,da
\end{align}
for all $d\geq 2$, $x,y\in\mathbb{S}^d$, and $r\in[0,\pi/2]$, where $r(a)=\arccos\bigl(\cos r/\sqrt{1-a^2}\bigr)$, and $x',y'$ are arbitrary points in $\mathbb{S}^{d-1}$ satisfying $d_{\mathbb{S}^d}(x,y)=d_{\mathbb{S}^{d-1}}(x',y')$ (there appears to be a misprint in \cite{Est10} regarding formula \eqref{eq:Estrade}). This recursion is particularly useful if the balls are hemispheres, i.e.\ $r=\pi/2$, yielding for all $d\geq 1$
\begin{align*}
P\bigl(x,y\in B_{\pi/2}(X)\bigr)= \frac{1}{2}-\frac{d_{\mathbb{S}^d}(x,y)}{2\pi},\quad x,y\in\mathbb{S}^d.
\end{align*}

From these formulae it is possible to compute $P(x,y\in B)$ for a discretely distributed radius $R$, although the formulae become quickly lengthy. In what follows we consider a family of continuous distributions for $R$ which results in rather simple formulae for $P(x,y\in B)$. A hyperplane in $\R^{d+1}$ that intersects $\mathbb{S}^d$ divides $\mathbb{S}^d$ into two spherical caps. If $r\in[0,\pi]$ is the radius of one such spherical cap, the distance of the hyperplane to the origin is given by the absolute value of $\cos r$. We assume henceforth, that $\cos R$ is continuously distributed with a distribution function of the form
\begin{align}\label{eq: Vertfkt cos R}
F_Q(t)=\biggl(\frac{1}{2}+\sum_{q=0}^Q p_q t^{2q+1}\biggr)\1_{[-1,1]}(t)+\1_{(1,\infty)}(t),\quad t\in\R,
\end{align}
for $Q\in\N_0$ and $p_0,\dots,p_Q\in\R_+$ with $\sum_{q=0}^Q p_q=1/2$. If $Q=0$ and $p_0=1/2$, this is the distribution function of the uniform distribution on $[-1,1]$. 

\begin{proposition}\label{example pxy formel}
Assume that $\cos R$ is continuously distributed with the distribution function $F_Q$ given in \eqref{eq: Vertfkt cos R} and set $d_{xy}=d_{\mathbb{S}^d}(x,y)$. Then for all $d\geq 1$ and all $x,y\in\mathbb{S}^d$
\begin{align}\label{eq: pxy formel mit cos R}
P(x,y\in B)=\frac{1}{2}-\sum_{q=0}^Q\sum_{l=1}^{q+1}p_q C_{q,l,d}\sin^{2l-1}\frac{d_{xy}}{2}\cos^{2(q-l+1)}\frac{d_{xy}}{2}
\end{align}
with
\begin{align*}
C_{q,l,d}=2^{-(2q+1)}\frac{\Gamma(2q+2)\Gamma((d+1)/2)}{\Gamma((2l+1)/2)\Gamma(q-l+2)\Gamma((2q+d+2)/2)}
\end{align*}
holds true.
\end{proposition}
\begin{proof}
The distribution function $F_Q$ in \eqref{eq: Vertfkt cos R} fulfills $F_Q(t)+F_Q(-t)=1$ for all $t\in\R$, which is equivalent to $\cos R\overset{d}{=}-\cos{R}$ or $R\overset{d}{=}\pi-R$. With the symmetry of $X$ and the definition of $d_{\mathbb{S}^d}$ this gives for all $x\in\mathbb{S}^d$
\begin{align*}
P(d_{\mathbb{S}^d}(x,X)\leq R)=P(d_{\mathbb{S}^d}(x,-X)\leq R)=P(d_{\mathbb{S}^d}(x,X)\geq \pi-R)=P(d_{\mathbb{S}^d}(x,X)> R)
\end{align*}
and consequently $P(x\in B)=1/2$. Thus 
\begin{align}\label{eq: pxy formel im beweis sphare}
P(x,y\in B)=\frac{1}{2}-P(x\in B,y\notin B)=\frac{1}{2}-P\bigl(d_{\mathbb{S}^d}(x,X)\leq R < d_{\mathbb{S}^d}(y,X)\bigr).
\end{align}
The surface measure \eqref{eq:surface measure sphere} is rotational invariant and we can therefore replace $x$ and $y$ in \eqref{eq: pxy formel im beweis sphare} by any points $x_+,x_-\in\mathbb{S}^d$, which satisfy $d_{\mathbb{S}^d}(x_+,x_-)=d_{xy}$. A convenient choice is
\begin{align}\label{eq:wahl von xplusminus}
x_{\pm}=\Phi_d\biggl(\pi\mp\frac{\pi}{2},\frac{\pi}{2},\dots,\frac{\pi}{2},\frac{d_{xy}}{2}\biggr)=\biggl(0,\pm\sin\frac{d_{xy}}{2},0,\dots,0,\cos\frac{d_{xy}}{2}\biggr).
\end{align}
By independence of $X$ and $R$ we have
\begin{align*}
P\bigl(d_{\mathbb{S}^d}(x_+,X)&\leq R < d_{\mathbb{S}^d}(x_-,X)\bigr)\\
&=\int_{\mathbb{S}^d}P\bigl(d_{\mathbb{S}^d}(x_+,z)\leq R < d_{\mathbb{S}^d}(x_-,z)\bigr)\1_{d_{\mathbb{S}^d}(x_+,z)\leq d_{\mathbb{S}^d}(x_-,z)}\,d\hat{\sigma_d}(z)\\
&=\int_{\mathbb{S}^d}P\bigl(\langle x_-,z\rangle< \cos R \leq  \langle x_+,z\rangle\bigr)\1_{\langle x_+-x_-,z\rangle\geq 0}\,d\hat{\sigma_d}(z)\\
&=\sum_{q=0}^Q p_q\int_{\mathbb{S}^d}\bigl(\langle x_+,z\rangle^{2q+1} -  \langle x_-,z\rangle^{2q+1}\bigr)\1_{\langle x_+-x_-,z\rangle\geq 0}\,d\hat{\sigma_d}(z).
\end{align*}
Passing to spherical coordinates \eqref{eq:sphaerische koordinaten}, the difference $\langle x_+,z\rangle^{2q+1} -  \langle x_-,z\rangle^{2q+1}$ becomes
\begin{align*}
\sum_{l=0}^{2q+1}&\binom{2q+1}{l}\bigl(1-(-1)^{l}\bigr)\sin^l\frac{d_{xy}}{2}\sin^l\varphi\prod_{i=1}^{d-1}\sin^l\theta_i\cos^{2q+1-l}\frac{d_{xy}}{2}\cos^{2q+1-l}\theta_{d-1}\\
&\hspace{-8pt}= 2\sum_{l=1}^{q+1}\binom{2q+1}{2l-1}\sin^{2l-1}\frac{d_{xy}}{2}\sin^{2l-1}\varphi\prod_{i=1}^{d-1}\sin^{2l-1}\theta_i\cos^{2(q-l+1)}\frac{d_{xy}}{2}\cos^{2(q-l+1)}\theta_{d-1}.
\end{align*}
Furthermore, the condition $\langle x_+-x_-,z\rangle\geq 0$ becomes in spherical coordinates $\varphi\in[0,\pi]$, and this in fact explains the choice of $x_{\pm}$ in \eqref{eq:wahl von xplusminus}. Altogether we obtain
\begin{align*}
P&\bigl(d_{\mathbb{S}^d}(x_+,X)\leq R < d_{\mathbb{S}^d}(x_-,X)\bigr)\\
&=\sum_{q=0}^Q\sum_{l=1}^{q+1} p_q\frac{\Gamma((d+1)/2)}{\pi^{(d+1)/2}}\binom{2q+1}{2l-1}\sin^{2l-1}\frac{d_{xy}}{2}\cos^{2(q-l+1)}\frac{d_{xy}}{2}\\
&\hspace{0.95em}\times \int_0^{\pi}\sin^{2l-1}\varphi \,d\varphi\prod_{i=1}^{d-2}\int_0^{\pi}\sin^{2l-1+i}\theta_i\, d\theta_i\int_0^{\pi}\sin^{2l+d-2}\theta_{d-1}\cos^{2(q-l+1)}\theta_{d-1}\,d\theta_{d-1}.
\end{align*}
For the last integral we can use formulae $3.621.5$ and $8.384.1$ in \cite{Ry65} because the exponent of the cosine is even. The other integrals can be evaluated with formulae $3.621.1, 8.384.1$, and $8.335.1$ in \cite{Ry65}. We get
\begin{align*}
P&\bigl(d_{\mathbb{S}^d}(x_+,X)\leq R < d_{\mathbb{S}^d}(x_-,X)\bigr)\\
&\hspace{-3pt}=\sum_{q=0}^Q\sum_{l=1}^{q+1} p_q\binom{2q+1}{2l-1}\frac{\Gamma((d+1)/2)\Gamma(l)\Gamma((2q-2l+3)/2)}{\pi\Gamma((2q+d+2)/2)}\sin^{2l-1}\frac{d_{xy}}{2}\cos^{2(q-l+1)}\frac{d_{xy}}{2}.
\end{align*}
Writing the binomial coefficient in terms of the Gamma function and using $8.335.1$ in \cite{Ry65} two times we find formula \eqref{eq: pxy formel mit cos R}.
\end{proof}

For example if $d=2$, $Q=0$, and $p_0=1/2$,  i.e.\ the random variable $\cos R$ is uniformly distributed on $[-1,1]$, Proposition \ref{example pxy formel} yields
\begin{align*}
P(x,y\in B)=\frac{1}{2}-\frac{1}{4}\sin\frac{d_{\mathbb{S}^2}(x,y)}{2},
\end{align*}
while the choice $Q=1$, $p_0=0$, and $p_1=1/2$, results in 
\begin{align}\label{eq: pxy formel etwas kompl}
P(x,y\in B)=\frac{1}{2}-\frac{3}{16}\sin\frac{d_{\mathbb{S}^2}(x,y)}{2}\cos^2\frac{d_{\mathbb{S}^2}(x,y)}{2}-\frac{1}{8}\sin^3\frac{d_{\mathbb{S}^2}(x,y)}{2}.
\end{align}
Plugging in these formulae in the formulae of Corollary \ref{KorollarmitSPezialfallen}, we get correlation functions of submodels of the mosaic random field \eqref{allgFeld} on $\mathbb{S}^2$. More examples for the two-dimensional sphere can be found in table \ref{table sphere}.

\begin{sidewaystable}
\begin{center}
\renewcommand{\arraystretch}{1.5}
\begin{tabular}{l l l l l l l}
\hline
Correlation functions on $\mathbb{S}^2$  & Model & $R$ & $N$ & $U$ & Lit.\\ \hline 
1. $\rho(x,y)=e^{-\left(d_{xy}/c\right)^{\alpha}}$ & SM & $\frac{\pi}{2}$ & $\sum_{l=1}^{L_1} K_l$ & & \cite{Gnei13} \\
2. $\rho(x,y)=\bigl(1+\bigl(\frac{d_{xy}}{c}\bigr)^{\alpha}\bigr)^{-\beta/\alpha}$ & SM & $\frac{\pi}{2}$ & $\sum_{l=1}^{L_2}K_l$ & & \cite{Gnei13} \\
3. $\rho(x,y)=\lambda_1\bigl(1-\frac{d_{xy}}{\pi}\bigr)+(1-\lambda_1)$ & RT & $\frac{\pi}{2}$ & $\operatorname{Geo}\bigl(\frac{\lambda_1}{2(2-\lambda_1)}\bigr)$ & $\mathcal{N}\bigl(\frac{1}{\sqrt{2}},\frac{1}{2}\bigr)$ & \cite{Gnei13}* \\
4. $\rho(x,y)=\lambda_2\bigl(1-\frac{d_{xy}}{\pi}\bigr)e^{-d_{xy}/c}+(1-\lambda_2)\bigl(1-\frac{d_{xy}}{\pi}\bigr)$ & MRT & $\frac{\pi}{2}$ & $\operatorname{Poi}\bigl(\frac{\pi}{c}\bigr)$ & $\mathcal{N}\bigl(1,\frac{\lambda_2}{1-\lambda_2}\bigr)$ &\cite{Gnei13}* \\
5. $\rho(x,y)=1-2^{1-\alpha}\frac{d_{xy}/\pi}{(1+d_{xy}/\pi)^{1-\alpha}}$ & DL & $\frac{\pi}{2}$ & $K$ &  & \eqref{eq:dead leaves corr} \\
6. $\rho(x,y)=\1_{d_{xy}=0}+\frac{1}{\pi(1-\cos r)}\Bigl(\arccos\frac{\cos^2r-\cos d_{xy}}{\sin^2 r}$ & RT & $r\in(0,\frac{\pi}{2}]$ & $\operatorname{Poi}$ & & \cite{Tov01} \\
\hspace{\mylengthzwei}$-2\cos r\arccos\frac{\cos r(1-\cos d_{xy})}{\sin r\sin d_{xy}}\Bigr)\1_{0<d_{xy}\leq 2r}$  \\
7. $\rho(x,y)=e^{-\left(\sin(d_{xy}/2)/c\right)^{\alpha}}$ & SM & $\cos R\sim\mathcal{U}\bigl([-1,1]\bigr)$ & $\sum_{l=1}^{L_3} K_l$ & & \cite{Gnei13, Yad83} \\
8. $\rho(x,y)=\bigl(1+\bigl(\frac{1}{c}\sin\frac{d_{xy}}{2}\bigr)^{\alpha}\bigr)^{-\beta/\alpha}$ & SM & $\cos R\sim\mathcal{U}\bigl([-1,1]\bigr)$ & $\sum_{l=1}^{L_4} K_l$ & & \cite{Gnei13, Yad83} \\
9. $\rho(x,y)=\lambda_1\bigl(1-\frac{1}{2}\sin\frac{d_{xy}}{2}\bigr)+(1-\lambda_1)$ & RT & $\cos R\sim\mathcal{U}\bigl([-1,1]\bigr)$ & $\operatorname{Geo}\bigl(\frac{\lambda_1}{2(2-\lambda_1)}\bigr)$ & $\mathcal{N}\bigl(\frac{1}{\sqrt{2}},\frac{1}{2}\bigr)$ & \cite{Gnei13}* \\
10. $\rho(x,y)=1-\frac{1}{4}\sin^3\frac{d_{xy}}{2}-\frac{3}{8}\sin\frac{d_{xy}}{2}\cos^2\frac{d_{xy}}{2}$ & RT & see \eqref{eq: Vertfkt cos R} & $\operatorname{Poi}$ & & \eqref{eq: pxy formel etwas kompl}  \\
11. $\rho(x,y)=1-\bigl(\frac{1}{2}\sin\frac{d_{xy}}{2}\bigr)^{\alpha}$ & SM & $\cos R\sim\mathcal{U}\bigl([-1,1]\bigr)$ & $K$ & & \cite{Gnei13}*\\
\hline
\renewcommand{\arraystretch}{1}
\end{tabular}
\caption{Examples of correlation functions of mosaic random fields on $\mathbb{S}^2$. Here $x$ and $y$ are points in $\S$ and $d_{xy}=d_{\mathbb{S}^2}(x,y)$ is the great circle distance. The random sets are closed balls of radius $R$ centered at $X$. The distribution of the random variables are as follows: $X\sim\mathcal{U}(\mathbb{S}^2)$, $L_1\sim\operatorname{Poi}\bigl((\pi/c)^{\alpha}\bigr)$, $L_2\sim\operatorname{NegBin}\bigl(\beta/\alpha,(1+(\pi/c)^{\alpha})^{-1}\bigr)$, $L_3\sim\operatorname{Poi}\bigl((2/c)^{\alpha}\bigr)$, $L_4\sim\operatorname{NegBin}\bigl(\beta/\alpha,(1+(2/c)^{\alpha}\bigr)^{-1})$, the random variables $K,K_l$ are i.i.d.\ with the probability generating function defined in \eqref{eq:pgf alpha power}. The parameters have the following range: $\alpha\in(0,1]$, $\beta,c>0$, $\lambda_1\in(0,1]$, $\lambda_2\in(0,1)$. Wherever an entry is left blank the distribution of the corresponding random variable is arbitrary. The symbol $\operatorname{Poi}$ indicates, that the corresponding random variable is Poisson distributed, but the parameter of the Poisson distribution is arbitrary. A \textquotesingle *\textquotesingle{} at the reference indicates that the given correlation function is new, but can be obtained as convex combinations or products of known correlation functions.}\label{table sphere}
\end{center}
\end{sidewaystable}

\section{Cylinder and Torus}\label{sec:Torus}
To conclude we give a short excursion to two more exotic spaces, cylinder and torus. Let $O=\mathbb{S}^1\times [0,h]$ be an open cylinder with radius $1$ and height $h>0$. Let $d_{\mathbb{S}^1}(x_1,y_1)=\arccos\langle x_1, y_1\rangle$, $x_1,y_1\in\mathbb{S}^1$, be the geodesic metric on $\mathbb{S}^1$. Then 
\begin{align*}
d_O\bigl((x_1,x_2),(y_1,y_2)\bigr)=\sqrt{d_{\mathbb{S}^1}^2(x_1,y_1)+|x_2-y_2|^2},\quad x_1,y_1\in\mathbb{S}^1,\, x_2,y_2\in[0,h],
\end{align*}
defines a metric on $O$. Fix $a\in(0,\pi]$, and let $(D_n,n\in\N)$ be a sequence of $[0,a]$-valued random variables, let $(U_n,n\in\N)$ be a sequence of uniformly distributed random variables on $[0,2\pi)$, and let $(V_n,n\in\N)$ be a sequence of uniformly distributed random variables on $[-a/2, h+a/2]$. Suppose all random variables above are independent. For $F(u,v)=\bigl(\cos u, \sin u, v\bigr)$, $u\in[0,2\pi)$, $v\in[-a/2,h+a/2]$, define $(X_n,n\in\N)$ by $X_n=F(U_n,V_n)$, $n\in\N$. Then $B_n=B_{D_n/2}(X_n)=\{z\in O : d_O(z,X_n)\leq D_n/2\}$ defines an i.i.d.\ sequence $(B_n,n\in\N)$ of random closed balls on $O$. Let $D$ be equal to some constant $t\in[0,a]$, such that a single ball does not intersect itself. Then
\begin{align*}
P(x,y\in B)&=P\bigl(X\in B_{t/2}(x)\cap B_{t/2}(y)\bigr)=\frac{\lambda^2\bigl(F^{-1}\bigl(B_{t/2}(x)\cap B_{t/2}(y)\bigr)\bigr)}{2\pi(h+a)}
\end{align*}
holds for all $x,y\in O$. The set $F^{-1}\bigl(B_{t/2}(x)\cap B_{t/2}(y)\bigr)$ is the intersection of two balls $B_{t/2}(\tilde{x})$ and $B_{t/2}(\tilde{y})$ in $\R^2$ with $\tilde{x},\tilde{y}\in [-t/2,2\pi+t/2)\times [0,h]$ and $\|\tilde{x}-\tilde{y}\|=d_O(x,y)$, where a part of this intersection which possibly exceeds the left or right boundary of $[0,2\pi)\times [-a/2,h+a/2]$ is reflected to the opposite side. The volume of the intersection of the balls does not change by this reflection and hence we can use \eqref{eq: Tovchi formel} to get
\begin{align}\label{eq:formel cylinder}
P(x,y\in B)=\frac{1}{4\pi(h+a)}\Bigl(t^2\arccos\frac{d_O(x,y)}{t}-d_O(x,y)\sqrt{t^2-d_O^2(x,y)}\Bigr)\1_{d_O(x,y)\leq t}.
\end{align}
Just as in section \ref{sec:Rd}, this formula can be integrated with respect to the distribution of $D$ in order to obtain $P(x,y\in B)$ for a random diameter of the ball $B$ and the resulting expressions are up to the normalization constant equal to \eqref{eq:spherical correlation} and \eqref{eq: R^d uniform radius} with $d_{xy}=d_O(x,y)$. Using for example the distribution function \eqref{eq: Alfaros distribution function} of Sironvalle for the diameter of the balls, we get for the corresponding random token field \eqref{RCfeld} with Poisson distributed number of balls the correlation function
\begin{align*}
\rho(x,y)=\Biggl(&1-\frac{3\sqrt{d_{\mathbb{S}^1}^2(x_1,y_1)+|x_2-y_2|^2}}{2a}\\&\hspace{4em}+\frac{(d_{\mathbb{S}^1}^2(x_1,y_1)+|x_2-y_2|^2)^{3/2}}{2a^3}\Biggr)\1_{\sqrt{d_{\mathbb{S}^1}^2(x_1,y_1)+|x_2-y_2|^2}\leq a}.
\end{align*}

The two-dimensional torus $\mathbb{T}^2=\mathbb{S}^1\times\mathbb{S}^1$ can be treated in a similar way. A convenient choice for the random closed ball here is $B=B_{D/2}(X)=\{z\in\mathbb{T}^2: d_{\mathbb{T}^2}(z,X)\leq D/2\}$ with $d_{\mathbb{T}^2}(x,y)=\sqrt{d_{\mathbb{S}^1}^2(x_1,y_1)+d_{\mathbb{S}^1}^2(x_2,y_2)}$, $x,y\in\mathbb{T}^2$. Here, we let $X=F(U,V)$ with the parametrization $F(u,v)=\bigl(\cos u,\sin u,\cos v,\sin v\bigr)$, $u,v\in[0,2\pi)$, the random variables $U$ and $V$ are uniformly distributed on $[0,2\pi)$, the diameter $D$ is a $[0,a]$-valued random variable for a cutoff $a\in[0,\pi]$, and all random variables are assumed to be independent.  For example, if the diameter $D$ is equal to some constant $t\in[0,a]$ and $x,y\in \mathbb{T}^2$, then the probability $P(x,y\in B)$ for $x,y\in \mathbb{T}^2$ is given by \eqref{eq:formel cylinder} where $d_{\mathbb{T}^2}$ replaces $d_{O}$ and the normalization constant is $1\big/\bigl(8\pi^2\bigr)$.

\vspace*{.5\baselineskip}\noindent
\textbf{Acknowledgement.} The authors gratefully acknowledge support by Deutsche For\-schungs\-ge\-mein\-schaft through the Research Training Group RTG 1953.


%
%
%
%

\bibliography{A_general_class_of_mosaic_random_fields}
\bibliographystyle{apt}

\end{document}